\UseRawInputEncoding
\documentclass[12pt, reqno]{amsart}
\usepackage[margin=1in]{geometry}
\usepackage{amssymb,latexsym,amsmath,amscd,amsfonts}
\usepackage{latexsym}
\usepackage[mathscr]{eucal}
\usepackage{bm}
\usepackage{mathptmx}

\def\textmatrix#1&#2\\#3&#4\\{\bigl({#1 \atop #3}\ {#2 \atop #4}\bigr)}
\def\dispmatrix#1&#2\\#3&#4\\{\left({#1 \atop #3}\ {#2 \atop #4}\right)}
\newcommand{\beg}{\begin{equation}}
	\newcommand{\eeg}{\end{equation}}
\newcommand{\ben}{\begin{eqnarray*}}
	\newcommand{\een}{\end{eqnarray*}}

\newtheorem{thm}{Theorem}[section]
\newtheorem{cor}[thm]{Corollary}
\newtheorem{lem}[thm]{Lemma}

\newtheorem{prop}[thm]{Proposition}
\numberwithin{equation}{section} 
\theoremstyle{definition}
\newtheorem{defn}[thm]{Definition}

\newcommand{\HS}{\mathcal H}
\newcommand{\C}{\mathbb{C}}
\newcommand{\Ccc}{\mathbb{C}^3}

\newcommand{\D}{\mathbb{D}}
\newcommand{\T}{\mathbb{T}}
\newcommand{\ft}{\mathcal F_O}

\newcommand{\E}{\mathbb E}

\newcommand{\ov}{\overline}

\newcommand{\la}{\langle}
\newcommand{\ra}{\rangle}

\begin{document}
	\title[Triangular $\E$-contraction, factorization of contractions and subvarieties]
	{Triangular Tetrablock-contractions, factorization of contractions, dilation and subvarieties}
	
	\author[Sourav Pal]{Sourav Pal}
	\address[Sourav Pal]{Mathematics Department, Indian Institute of Technology Bombay, Powai, Mumbai - 400076, India.}
	
	\email{souravmaths@gmail.com , sourav@math.iitb.ac.in}

	
	
	\thanks{The author is supported by the Seed Grant of IIT Bombay, the CPDA of the Govt. of India and the MATRICS Award of SERB, (Award No. MTR/2019/001010) of DST, India.}

\begin{abstract}

A commuting triple of Hilbert space operators $(A,B,P)$, for which the closed tetrablock $\overline{\mathbb E}$ is a spectral set, is called a \textit{tetrablock-contraction} or simply an $\mathbb E$-\textit{contraction}, where
\[
\mathbb E=\{(a_{11},a_{22}, \det A):\, A=[a_{ij}]\in \mathcal M_2(\mathbb C), \; \|A\| <1  \} \subset \mathbb C^3
\]
is a polynomially convex domain which is naturally associated with the $\mu$-synthesis problem. We introduce triangular $\E$-contractions and prove that every pure triangular $\E$-contraction dilates to a pure triangular $\E$-isometry. We construct a functional model for a pure triangular $\mathbb E$-isometry and apply that model to find a new proof for the famous Berger-Coburn-Lebow Model Theorem for commuting isometries. Next we give an alternative proof to the more generalized version of Berger-Coburn-Lebow Model, namely the factorization of a pure contraction due to Das, Sarkar and Sarkar (\textit{Adv. Math.} 322 (2017), 186 -- 200). We find a necessary and sufficient condition for the existence of $\mathbb E$-unitary dilation of an $\mathbb E$-contraction $(A,B,P)$ on the smallest dilation space and show that it is equivalent to the existence of a distinguished variety in $\mathbb E$ when the defect space $D_{P^*}$ is finite dimensional. 

\end{abstract}

\maketitle


\section{INTRODUCTION}

\vspace{0.2cm}
	
\noindent Throughout the paper, all operators are bounded linear operators acting on complex Hilbert spaces. A contraction is an operator with norm not greater than $1$. We denote by $\C, \D, \T$ the complex plane, the unit disk and the unit circle in the complex plane respectively with center at the origin. 

In this paper, we introduce and develop a theory based on a new class of operators associated with the tetrablock, namely the \textit{triangular tetrablock contractions}. The tetrablock $\mathbb E$, defined by
\[
\mathbb E=\{(a_{11},a_{22}, \det A):\, A=[a_{ij}]\in \mathcal M_2(\mathbb C), \; \|A\| <1  \} \subset \mathbb C^3,
\]
arises naturally in the $\mu$-synthesis problem. The $\mu$-synthesis is a part of the theory of robust control of
systems comprising interconnected electronic devices whose outputs
are linearly dependent on the inputs. Given a linear subspace $E$
of $\mathcal M_n(\mathbb C)$, the space of all $n \times n$
complex matrices, the functional
\[
\mu_E(A):= (\text{inf} \{ \|X \|: X\in E \text{ and } (I-AX)
\text{ is singular } \})^{-1}, \; A\in \mathcal M_n(\mathbb C),
\]
is called a \textit{a structured singular value}, where the linear
subspace $E$ is referred to as the \textit{structure}. If
$E=\mathcal M_n(\mathbb C)$, then $\mu_E (A)$ is equal to the
operator norm $\|A\|$, while if $E$ is the space of all scalar multiples of the identity matrix, then $\mu_E(A)$
is the spectral radius $r(A)$. For any linear subspace $E$ of
$\mathcal M_n(\mathbb C)$ that contains the identity matrix $I$,
$r(A)\leq \mu_E(A) \leq \|A\|$. For a detailed discussion on $\mu$-synthesis, an interested reader is referred to the pioneering
work of Doyle \cite{Doyle}. The aim of $\mu$-synthesis is to find an
analytic function $f$ from the open unit disk $\mathbb D$
to $\mathcal M_n(\mathbb C)$ subject to a finite
number of interpolation conditions such that $\mu_E(f(\lambda))<1$
for all $\lambda \in \mathbb D$. If $E$ is the linear subspace of
$2 \times 2$ diagonal matrices, then for any $A=(a_{ij}) \in
\mathcal M_2 (\mathbb C)$, $\mu_E (A)<1$ if and only if
$(a_{11},a_{22}, \det A)\in \mathbb E$ (see \cite{A:W:Y}, Section-9). In spite of having origin in the $\mu$-synthesis problem, the geometry and function theory of the tetrablock was first studied by Nicholas Young and his fellow collaborators, \cite{A:W:Y, young}. Later this domain has been extensively studied over past fifteen years by a quite a few mathematicians from both complex analytic and operator theoretic perspectives, see \cite{EZ, EKZ, Zwo, T:B, S:P-tetra1, S:P-tetra2, S:P-tetra3, Bh-Sau1, Ball-Sau} and the references therein.

\begin{defn}
A triple of commuting Hilbert space operators $(A,B,P)$ is called a \textit{tetrablock-contraction} or simply an $\E$-\textit{contraction} if the closed tetrablock $\ov{\E}$ is a spectral set for $(A,B,P)$, that is, the Taylor joint spectrum $\sigma_T(A,B,P) \subseteq \ov{\E}$ and von Neumann inequality
\[
\|f(A,B,P)\|\leq \sup_{z\in\ov{\E}} |f(z)|=\|f\|_{\infty,\; \ov{\E}} 
\]
holds for every rational function $f=\dfrac{p}{q}$ with $p,q \in \C[z_1,z_2,z_3]$, such that $q$ does not have any zero inside $\ov{\E}$. Here $f(A,B,P)=p(A,B,P)q(A,B,P)^{-1}$. Moreover, an $\E$-contraction $(A,B,P)$ is called \textit{pure} or $C._0$ if $P$ is a $C._0$ contraction that is ${P^*}^n \rightarrow 0$ strongly as $n \rightarrow \infty$.
\end{defn}

A point of the form $(z_1,z_2,z_1z_2)$, where $z_1,z_2$ are from the unit disk $\D$, always belongs to the tetrablock and is called a \textit{triangular point}, (see \cite{A:W:Y}). A straight-forward analogue of this result is that $(P,Q,PQ)$ is an $\E$-contraction, whence $P,Q$ are commuting contractions (see Lemma \ref{lem:triangular1}). Such an $\E$-contraction is called a \textit{triangular} $\E$-\textit{contraction}. We recall from literature special classes of $\E$-contractions, e.g. $\E$-unitary, $\E$-isometry etc. in Section \ref{sec:2}.

In Theorem \ref{thm:model-triang}, we construct a functional model for a pure triangular $\E$-isometry and by an application of that model we find in Theorem \ref{thm:B-C-Lb} a new proof for the following Berger-Coburn-Lebow model theorem for commuting isometries, which provides factorization of a pure isometry.

\begin{thm} [Berger-Coburn-Lebow, \cite{B-C-Lb}]
 
Let $V_1, \dots , V_n$ be commuting isometries on $\HS$ such that $V=\prod_{i=1}^n V_i$ is a pure isometry. Then, there exist projections $P_1, \dots , P_n$ and unitaries $U_1, \dots , U_n$ in $\mathcal B(\mathcal D_{V^*})$ such that
\[
(V_1, \dots , V_n , V) \equiv (T_{P_1^{\perp}U_1+ zP_1U_1 }, \dots , T_{P_n^{\perp}U_n+zP_nU_n }, T_z)  \;\; \text{ on } \; \; H^2(\mathcal D_{V^*}).
\] 

\end{thm}
In \cite{B-D-F1}, Bercovici, Douglas and Foias presented an independent proof of the Berger-Coburn-Lebow model theorem with a special emphasis on the case when $Ker\, \mathcal D_{V^*}$ is finite dimensional. Indeed, they gave a complete classification
of $n$-tuples for which $V$ is a pure isometry of multiplicity $n$. Recently Das, Sarkar and Sarkar generalized the factorization of Berger-Coburn-Lebow to the class of pure contractions (i.e. $C._0$ contractions) in the following way. 

\begin{thm} [Das-Sarkar-Sarkar, \cite{D:S:S}] \label{thm:factor}
Let $T$ be a pure contraction on a Hilbert space $\HS$. Then $T=T_1T_2$ for a pair of commuting contractions $T_1,T_2 \in \mathcal B(\HS)$ if and only if there exist a Hilbert space $E$, a projection $P$ and a unitary $U$ in $\mathcal B(E)$ such that $\HS \subseteq H^2(E)$ is a common invariant subspace for $(M_{\Phi}^*,M_{\Psi}^*)$ and the following identities hold:
\begin{itemize}

\item[(i)] $P_{\HS} M_{\Phi \Psi}|_{\HS}= P_{\HS} M_{\Psi \Phi}|_{\HS} = P_{\HS} M_{z}|_{\HS}$

\item[(ii)] $(T_1,T_2) \equiv (P_{\HS} M_{\Phi}|_{\HS}, P_{\HS} M_{\Psi}|_{\HS})$,

\end{itemize}
where $\Phi(z)=(P+zP^{\perp})U^*$ and $\Psi(z)=U(P^{\perp}+zP)$.

\end{thm}
In Theorem \ref{thm:triang-dilation}, we prove that every pure triangular $\E$-contraction dilates to a pure triangular $\E$-isometry. As a corollary of this dilation theorem we obtain an alternative proof for the above factorization theorem. Also, one can find some notable works due to Sau in this direction in \cite{Sau}.

Section \ref{sec:4} is devoted to showing an interplay between dilation of an $\E$-contraction and distinguished varieties in the tetrablock. A \textit{distinguished variety} in $\E$ is the intersection of an algebraic variety $W$ with $\E$ such that $W$ exits the tetrablock through its distinguished boundary without intersecting any other parts of its topological boundary. In Theorem \ref{thm:E-unitary dilation}, we provide a necessary and sufficient condition under which any $\E$-contraction $(A,B,P)$ admits a minimal $\E$-unitary dilation on the smallest dilation space. Then we show in Theorems \ref{thm:dilation-variety1} \& \ref{thm:dilation-variety2} that the existence of such a minimal $\E$-unitary dilation is equivalent to the existence of a distinguished variety in $\E$ when $\mathcal D_{P^*}$ is finite dimensional. This indeed gives a new characterization for the distinguished varieties in the tetrablock. Section \ref{sec:2} deals with a few results from the literature.

\vspace{0.2cm}

	\section{DEFINITIONS, TERMINOLOGIES AND A BRIEF LITERATURE} \label{sec:2}
	
	\vspace{0.2cm}
	
\noindent  Here we recollect a few results from the literature  and these results will be used frequently. We begin with the following scalar characterization theorem.

\begin{thm}[\cite{A:W:Y}, Theorems 2.2 \& 2.4]\label{thm:21}
Let $(x_1,x_2,x_3)\in \Ccc$. Then the following are equivalent.

\begin{enumerate}

\item $(x_1,x_2,x_3) \in \E \quad (\text{respectively, } \in \overline{\mathbb E})$ ;

\item $|x_1-\bar x_2 x_3|+|x_1x_2-x_3|< 1-|x_2|^2 \;$ (respectively, $\leq 1-|x_2|^2 $ and if $x_1x_2=x_3$ then, in addition $|x_1|\leq 1)$ ;
 
 \item $|x_2-\bar x_1 x_3|+|x_1x_2-x_3|< 1-|x_1|^2 \;$ (respectively, $\leq 1-|x_1|^2 $ and if $x_1x_2=x_3$ then, in addition $|x_2|\leq 1)$ ;
 
 \end{enumerate}	

\end{thm}
The distinguished boundary of the tetrablock was determined in \cite{A:W:Y} to be the following set:
\[
b\ov{\E}=\{(x_1,x_2,x_3)\in \ov{\E}\,:\, |x_3|=1 \}.
\]
Operator theory on the tetrablock was introduced in \cite{T:B} and was further developed in \cite{B-P1}, \cite{S:P-tetra2}, \cite{S:P-tetra3} and in a few more articles. Unitaries, isometries, co-isometries etc. are special classes of contractions. There are natural analogues of these classes for $\E$-contractions in the literature (see \cite{T:B}).

\begin{defn}
Let $A,B,P$ be commuting operators on $\mathcal H$. Then $(A,B,P)$
is called
\begin{itemize}
\item[(i)] an $\mathbb E$-\textit{unitary} if $A,B,P$ are normal
operators and the Taylor joint spectrum $\sigma_T (A,B,P)$ is a
subset of $b\mathbb{E}$ ;

\item[(ii)] an $\mathbb E$-isometry if there exists a Hilbert space
$\mathcal K \supseteq \mathcal H$ and an $\mathbb E$-unitary
$(Q_1,Q_2,V)$ on $\mathcal K$ such that $\mathcal H$ is a joint
invariant subspace of $A,B,P$ and that $(Q_1|_{\mathcal H},
Q_2|_{\mathcal H},V|_{\mathcal H})=(A,B,P)$ ;

\item[(iii)] an $\mathbb E$-co-isometry if the adjoint $(A^*,
B^*,P^*)$ is an $\mathbb E$-isometry.
\end{itemize}

\end{defn}
The following theorem provides a set of independent descriptions of the $\E$-unitaries.

\begin{thm}[{\cite{T:B}, Theorem 5.4}]\label{thm:tu}
    Let $\underline N = (N_1, N_2, N_3)$ be a commuting triple of
    bounded operators. Then the following are equivalent.

    \begin{enumerate}

        \item $\underline N$ is an $\mathbb E$-unitary,

        \item $N_3$ is a unitary and $\underline N$ is an $\mathbb
        E$-contraction,

        \item $N_3$ is a unitary, $N_2$ is a contraction and $N_1 = N_2^*
        N_3$.
    \end{enumerate}
\end{thm}
In the following result, we have a set of characterizations for the $\E$-isometries.

\begin{thm}[{\cite{T:B}, Theorem 5.7}] \label{thm:ti}

    Let $\underline V = (V_1, V_2, V_3)$ be a commuting triple of
    bounded operators. Then the following are equivalent.

    \begin{enumerate}

        \item $\underline V$ is an $\mathbb E$-isometry.

        \item $V_3$ is an isometry and $\underline V$ is an $\mathbb
        E$-contraction.

        \item $V_3$ is an isometry, $V_2$ is a contraction and $V_1=V_2^*
        V_3$.
        
        \item (Wold decomposition) There is an orthogonal decomposition $\HS =\HS_1 \oplus \HS_2$ such that $\HS_1 , \HS_2$ are common reducing subspaces for $V_1,V_2,V_3$ and that $(V_1|_{\HS_1}, V_2|_{\HS_1},V_3|_{\HS_1})$ is an $\E$-unitary and $(V_1|_{\HS_2},V_2|_{\HS_2},V_3|_{\HS_2})$ is a pure $\E$-isometry.

    \end{enumerate}
\end{thm}

Note that the canonical decomposition of an $\E$-contraction $(A,B,P)$ acting on a Hilbert space $\HS$ (see \cite{S:P-tetra3}) splits $\HS$ into two orthogonal parts $\HS=\HS_1 \oplus \HS_2$ such that $(A|_{\HS_1}, B|_{\HS_1},P|_{\HS_1})$ is an $\E$-unitary and $(A|_{\HS_2},B|_{\HS_2},P|_{\HS_2})$ is an $\E$-contraction with $P|_{\HS_2}$ being a c.n.u. contraction. This fact along with the Wold decomposition of an $\E$-isometry lead to defining the following natural classes of $\E$-contractions.

\begin{defn}\label{def3}
    Let $(A,B, P)$ be an $\mathbb E$-contraction on a Hilbert space $\mathcal
    H$. We say that $(A,B,P)$ is
    \begin{itemize}
        \item[(i)] a \textit{completely non-unitary} $\E$-\textit{contraction}, or simply a \textit{c.n.u.} $\mathbb E$-\textit{contraction} if $P$ is a
        c.n.u. contraction ;
        
        \item[(ii)] a $C._0$ or \textit{pure} $\E$-\textit{contraction} if $P$ is a $C._0$ contraction, i.e. ${P^*}^n \rightarrow 0$ strongly as $n \rightarrow \infty$.
        \end{itemize}

\end{defn}

The next theorem plays central role in the extensively studied operator theory of the tetrablock. See \cite{T:B, B-P1, S:P-tetra1, S:P-tetra2, S:P-tetra3} for further details.

\begin{thm}[\cite{T:B}, Theorem 1.3] \label{exist-tetra}
To every $\mathbb
E$-contraction $(A, B, P)$ there were two unique operators $F_1$
and $F_2$ on $\mathcal{\mathcal{D}_P} =\overline{\text{Ran}}(I -
P^*P)$ that satisfied the fundamental equations, i.e,
\[
A-B^*P = D_PF_1D_P\,, \qquad B-A^*P = D_PF_2D_P.
\]
\end{thm}
The operators $F_1,F_2$ are called the \textit{fundamental
operators} of $(A, B, P)$.

\vspace{0.4cm}

\section{THE TRIANGULAR $\E$-CONTRACTIONS, DILATION AND FACTORIZATION OF CONTRACTIONS} \label{sec:3}

\vspace{0.4cm}

\noindent We learned from Theorem \ref{thm:21} that if $x_1, x_2$ is in $\ov{\D}$, then $(x_1, x_2, x_1x_2)$ belongs to $\ov{\E}$ and such triples are known as \textit{triangular points} of the closed tetrablock. The set of triangular points play important role in determining several characterizations of the automorphisms of the tetrablock, e.g. see \cite{A:W:Y}. Here we shall study the $\E$-contractions which are triangular, i.e. of the form $(P,Q,PQ)$ for a pair of commuting contractions $P,Q$. We give an alternative proof to the famous Berger-Coburn-Lebow model theorem for commuting isometries and also find a different way of reaching the factorization of a $C._0$ contraction described in \cite{D:S:S}. Let us be armed with a few preparatory results. We begin with a simple result whose proof is a routine exercise.

\begin{lem}\label{lem:23}

If $X\subseteq \mathbb C^n$ is a polynomially convex set, then $X$
is a spectral set for a commuting tuple $(T_1,\dots,T_n)$ if and
only if

\begin{equation}\label{pT}
\|f(T_1,\dots,T_n)\|\leq \| f \|_{\infty,\, X}\,
\end{equation}
for all holomorphic polynomials $f$ in $n$-variables.

\end{lem}

The following elementary result will be of frequent use throughout the paper.

\begin{lem} [\cite{S:Pal-Decomp}, Lemma 3.3] \label{lem:E-1}
Let $A,B$ be commuting contractions on a Hilbert space $\HS$. Then $AB$ is unitary if and only if $A$ and $B$ are unitaries.
\end{lem}

\begin{cor}
For commuting contractions $P_1,\dots ,P_n$, $\prod_{i=1}^n P_i$ is unitary if and only if each of $P_1,\dots ,P_n$ is a unitary.
\end{cor}

\begin{cor}\label{cor:E-11}
For an $\E$-unitary $(A,B,P)$, if $P=AB$ then $A,B$ are unitaries.
\end{cor}

Now we recall from our previous work \cite{S:P-tetra1} a result that provides a model for a pure $\E$-isometry in terms of Toeplitz operators on a certain vectorial Hardy Hilbert space.

\begin{thm} [\cite{S:P-tetra1}, Theorem 3.3] \label{modelthm1}
Let $(\hat{T_1},\hat{T_2},\hat{T_3})$ be a pure $\mathbb
E$-isometry acting on a Hilbert space $\mathcal H$ and let
$A_1,A_2$ denote the corresponding fundamental operators. Then
there exists a unitary $U:\mathcal H \rightarrow H^2(\mathcal
D_{{\hat{T_3}}^*})$ such that
\[
\hat{T_1}=U^*T_{\varphi}U,\quad \hat{T_2}=U^*T_{\psi}U \textup{
and } \hat{T_3}=U^*T_zU,
\]
where $\varphi(z)= G_1^*+ zG_2,\,\psi(z)= G_2^*+zG_1, \quad
z\in\mathbb D$ and $G_1=UA_1U^*$ and $G_2=UA_2U^*$. Moreover,
$A_1,A_2$ satisfy
\begin{enumerate}
\item $[A_1,A_2]=0\,;$ \item $[A_1^*,A_1]=[A_2^*,A_2] \,;$ and
\item $\|A_1^*+zA_2\|\leq 1$ for all $z\in {\mathbb D}$.
\end{enumerate}
Conversely, if $A_1$ and $A_2$ are two bounded operators on a
Hilbert space $E$ satisfying the above three conditions, then
$(T_{A_1^*+zA_2},T_{A_2^*+zA_1},T_z)$ on $H^2(E)$ is a pure
$\mathbb E$-isometry.
\end{thm}

We now come to the study of triangular $\E$-contractions. We begin with the following lemma from \cite{S:Pal-Decomp}.

\begin{lem} [\cite{S:Pal-Decomp}, Lemma 3.2] \label{lem:triangular1}
If $P,Q$ are commuting contractions, then $(P,Q,PQ)$ is an $\E$-contraction.
\end{lem}

\begin{defn}
Let $(A,B,P)$ be an $\E$-contraction. Then $(A,B,P)$ is said to be
\begin{itemize}

\item[(i)] a \textit{triangular} $\E$-\textit{contraction} if $(A,B,P)=(P,Q,PQ)$ for a pair of commuting contractions $P,Q$.
 
 \item[(ii)] a \textit{triangular} $\E$-\textit{isometry} (or, a \textit{triangular} $\E$-\textit{unitary}) if $(A,B,P)$ is an $\E$-isometry (or an $\E$-unitary) which is of the form $(P,Q,PQ)$ for a pair of commuting contractions $P,Q$.
 
 \end{itemize}
 
\end{defn}

Of particular interest is the class of triangular $\E$-isometries and unitaries. The following lemma gives a clear description of them.

\begin{lem} \label{lem:char-triang1}
Let $A,B,P$ be a commuting triple of Hilbert space contractions. Then
\begin{itemize}

\item[(a)] $(A,B,P)$ is a triangular $\E$-unitary if and only if $(A,B,P)=(P,Q,PQ)$ for a pair of commuting unitaries $P,Q$;

\item[(b)] $(A,B,P)$ is a triangular $\E$-isometry if and only if $(A,B,P)=(P,Q,PQ)$ for a pair of commuting isometries $P,Q$.

\end{itemize}
\end{lem}

\begin{proof}

If $P,Q$ are commuting isometries (or unitaries), then it follows from part-(2) of Theorem \ref{thm:ti} (or Theorem \ref{thm:tu}) that $(P,Q,PQ)$ is a triangular $\E$-isometry (or a triangular $\E$-unitary). So, we prove the converse parts.

Let $(A,B,P)$ be a triangular $\E$-unitary. Then $(A,B,P)=(P,Q,PQ)$ for a pair of commuting contractions $P,Q$ and Lemma \ref{lem:E-1} shows that $P,Q$ are unitaries.

Suppose $(A,B,P)$ is a triangular $\E$-isometry. By Wold-decomposition as in Theorem \ref{thm:ti} and by the model Theorem \ref{modelthm1}, we have
\[
(A,B,P)=\left(
\begin{bmatrix}
A_{11} & 0 \\
0 & T_{F_1^*+zF_2}
\end{bmatrix},
\begin{bmatrix}
B_{11} & 0 \\
0 & T_{F_2^*+zF_1}
\end{bmatrix},
\begin{bmatrix}
P_{11} & 0 \\
0 & T_{z}
\end{bmatrix}
\right)\,,
\]
where $(A_{11}, B_{11},P_{11})$ is an $\E$-unitary and $(T_{F_1^*+zF_2},T_{F_2^*+zF_1},T_z)$ is a pure $\E$-isometry. By the previous part of the proof, we have that $A_{11}, B_{11}$ are unitaries as $A_{11}B_{11}=P_{11}$, and $P_{11}$ is a unitary. Again, since $(A,B,P)=(P,Q,PQ)$, we have that $T_{F_1^*+zF_2}T_{F_2^*+zF_1}=T_{F_1^*+zF_2}T_{F_2^*+zF_1}=T_z$, which further gives
\[
F_1F_2=F_2F_1=0 \; \; \& \;\; F_1^*F_1+F_2F_2^*=F_1F_1^*+F_2^*F_2=I.
\]
With these operator identities it can be easily verified that $T_{F_1^*+zF_2}$ and $T_{F_2^*+zF_1}$ are isometries and the proof is complete.

\end{proof}

It is evident from the Wold-decomposition of an $\E$-isometry (see Theorem \ref{thm:ti}) that every $\E$-isomerty splits into two orthogonal parts of which one is an $\E$-unitary and the other is a pure $\E$-isometry. We already have a concrete model for a pure $\E$-isometry in Theorem \ref{modelthm1}, though here we are curious about the more specific class --- the pure triangular $\E$-isometries, which are described below more explicitly.

\begin{thm} \label{thm:model-triang}

A commuting triple of operators $(V_1,V_2,V)$ with $V=V_1V_2$, acting on $\HS$ is a pure triangular $\E$-isometry if and only if $\HS$ is unitarily equivalent to $H^2(\mathcal D_{V^*})$  by a unitary, say $U$ and there exist a unitary $W$ and a projection $Q$ in $\mathcal B(\mathcal D_{V^*})$ such that $(V_1,V_2,V)$ is unitarily equivalent to $(T_{Q^{\perp}W+zQW}, T_{W^*Q+zW^*Q^{\perp}} , T_z)$ by the same unitary $U$, where $Q^{\perp}=I-Q$.

\end{thm}

\begin{proof}

Evidently a commuting triple of Toeplitz operators $(T_{Q^{\perp}W+zQW}, T_{W^*Q+zW^*Q^{\perp}} , T_z)$ is a pure $\E$-isometry and it can be easily verified that the product of the first two components is equal to $T_z$, which proves that it is triangular.

Conversely, suppose $(V_1,V_2,V)$ is a pure triangular $\E$-isometry. Then, by Theorem \ref{modelthm1} we have that the spaces $\HS$ and $H^2(\mathcal D_{V^*})$ can be identified by a unitary say $U$ and that $(V_1,V_2,V)$ can be identified with $(T_{F_1^*+zF_2}, T_{F_2^*+zF_1} , T_z)$ by the same unitary $U$, where $F_1,F_2$ are the fundamental operators of $(V_1^*,V_2^*,V^*)$. Thus, without loss of generality let us assume that
\[
(V_1,V_2,V)=(T_{F_1^*+zF_2}, T_{F_2^*+zF_1} , T_z).
\]
The fact that $V_1V_2=V_2V_1=V$ gives
$F_1F_2=F_2F_1=0$ and $F_1^*F_1+F_2F_2^*=F_1F_1^*+F_2^*F_2=I$. Therefore, $F_1^*+zF_2 \in \mathcal B(\mathcal D_{V^*})$ is a unitary for any $z\in \T$. Let $F_1^*+F_2=U_1$ and $F_1^*-F_2=U_2$. Then, $F_1^*=\dfrac{U_1+U_2}{2}$ and $F_2=\dfrac{U_1-U_2}{2}$. Form $F_1F_2=F_2F_1=0$, we have $U_2U_1^*=U_1U_2^*$ and $U_1^*U_2=U_2^*U_1$. Set $Q=\dfrac{I-U_2U_1^*}{2}$. Then, $Q$ is self-adjoint and using $U_1^*U_2=U_2^*U_1$ we have that $Q^2=Q$. Thus, $Q$ is a projection and evidently
\[
F_1^*=\dfrac{U_1+U_2}{2}=(I-Q)U_1 \; \; \; \& \; \; \; F_2= \dfrac{U_1-U_2}{2}=QU_1.
\]
We choose to denote $W=U_1$ and have our desired conclusion. The proof is now complete.

\end{proof}


\subsection{Berger-Coburn-Lebow Theorem and factorization of a pure contraction}
It is clear from the above theorem that the unitary part of $(V_1, \dots , V_n)$ consists of commuting unitaries and the product of the isometries in the non-unitary part is a pure isometry. So, it suffices to have a model for the non-unitary part which indeed was constructed in \cite{B-C-Lb} by Berger, Coburn and Lebow. We recall the result and give a new proof based on the functional model for pure triangular $\E$-isometries as in Theorem \ref{thm:model-triang}.

\begin{thm} [Berger-Coburn-Lebow, \cite{B-C-Lb}] \label{thm:B-C-Lb}

Let $V_1, \dots , V_n$ be commuting isometries on $\HS$ such that $V=\prod_{i=1}^n V_i$ is a pure isometry. Then, there exist projections $P_1, \dots , P_n$ and unitaries $U_1, \dots , U_n$ in $\mathcal B(\mathcal D_{V^*})$ such that
\[
(V_1, \dots , V_n , V) \equiv (T_{P_1^{\perp}U_1+ zP_1U_1 }, \dots , T_{P_n^{\perp}U_n+zP_nU_n }, T_z)  \;\; \text{ on } \; \; H^2(\mathcal D_{V^*}).
\] 
Moreover, $U_i^*P_i^{\perp} , P_iU_i$ are the fundamental operators of the $\E$-contraction $(V_i^*,V_i'^*,V^*)$ for each $i$.

\end{thm}

\begin{proof}
Here we are providing just an alternative and shorter proof. Evidently $(V_i, V_i', V)$ is a pure triangular $\E$-isometry for each $i$, where $V_i'=\prod_{j \neq i}V_j$. If we follow the proof of Theorem \ref{modelthm1} from \cite{S:P-tetra1}, we see that the unitary $U$ (as in Theorem \ref{modelthm1}) is unique so that $\hat T_3= U^*T_zU$ and that $(\hat{T}_1, \hat T_2 , \hat T_3)\equiv (T_{A_1^*+zA_2}, T_{A_2^*+zA_1}, T_z)$ by $U$, where $A_1,A_2$ are the fundamental operators of $(\hat T_1^*, \hat T_2^*, \hat T_3^*)$. Note that the third component i.e. $V$ is held fixed in all commuting triples $(V_i, V_i',V)$. Thus, there is a unitary $U:\HS \rightarrow H^2(\mathcal D_{V^*})$ such that
\[
(V_1, \dots , V_n , V_1', \dots , V_n', V) \equiv (T_{A_1^*+zB_1},\dots , T_{A_n^*+zB_n}, T_{B_1^*+zA_1}, \dots , T_{B_n^*+zA_n},  T_z)
\]
by the unitary $U$, where $A_i, B_i$ are the fundamental operators of $(V_i^*,V_i'^*,V^*)$ for $i=1, \dots , n$. By an application of Theorem \ref{thm:model-triang}, we arrive at the desired conclusion.
\end{proof}

We are now going to prove that every pure triangular $\E$-contraction dilates to a pure triangular $\E$-isometry and this is one of the pillars of this paper. Before that we recall the definitions of $\E$-isometric and $\E$-unitary dilations of an $\E$-contraction followed by a few useful results from the literature.

\begin{defn}
Let $(T_1,T_2,T_3)$ be a $\mathbb E$-contraction on $\mathcal H$.
A commuting tuple $(Q_1,Q_2,V)$ on $\mathcal K$ is said to be an
$\mathbb E$-isometric (or $\E$-unitary) dilation of $(T_1,T_2,T_3)$ if $\mathcal H
\subseteq \mathcal K$, $(Q_1,Q_2,V)$ is an $\mathbb E$-isometry (or $\E$-unitary)
and
$$ P_{\mathcal H}(Q_1^{m_1}Q_2^{m_2}V^n)|_{\mathcal H}=T_1^{m_1}T_2^{m_2}T_3^n,
\; \textup{ for all non-negative integers }m_1,m_2,n.
$$ Here $P_{\mathcal H}:\mathcal K \rightarrow \mathcal H$
is the orthogonal projection of $\mathcal K$ onto $\mathcal H$.
Moreover, the dilation is called {\em minimal} if
\[
\mathcal K=\overline{\textup{span}}\{ Q_1^{m_1}Q_2^{m_2}V^n h\,:\;
h\in\mathcal H \textup{ and }m_1,m_2,n\in \mathbb N \cup \{0\} \; \;(\text{ or } \in \mathbb Z \text{ respectively})\}.
\]
\end{defn}

\begin{prop} [\cite{S:P-tetra1}, Proposition 4.3] \label{prop:minimal-dil}
If an $\E$-contraction has an $\E$-isometric dilation, then it has a minimal $\E$-isometric dilation.
\end{prop} 

\begin{prop} [\cite{S:P-tetra1}, Proposition 4.4] \label{prop:dilation-extn}
Let $(A,B,P)$ on $\HS$ be an $\E$-contraction. Then $(Y_1,Y_2,Y)$ on $\mathcal K$ is a minimal $\E$-isometric dilation of $(A,B,P)$ if and only if $(Y_1^*,Y_2^*,Y^*)$ is an $\E$-co-isometric extension of $(A,B,P)$.
\end{prop}

\begin{thm}[\cite{S:P-tetra2}, Theorem 3.2] \label{thm:pure-E-dilation}
Let $(A,B,P)$ be a pure $\E$-contraction on a Hilbert space $\HS$ and let the fundamental operators $G_1,G_2$ of $(A^*,B^*,P^*)$ satisfy $[G_1,G_2]=0$ and $[G_1^*,G_1]=[G_2^*,G_2]$. Then $(R_1,R_2,V)$ acting on $\mathcal K = H^2 \otimes \mathcal D_{P^*}$, where
\[
R_1  = I \otimes G_1^* + M_z \otimes G_2 ,\; R_2  =I \otimes G_2^* + M_z \otimes G_1 , \; V = M_z \otimes I,
\]
is a minimal pure $\E$-isometric dilation of $(A,B,P)$.
\end{thm}

\begin{thm} \label{thm:triang-dilation}

Every triangular $\E$-contraction dilates to a triangular $\E$-isometry and every pure triangular $\E$-contraction dilates to a pure triangular $\E$-isometry.

\end{thm}

\begin{proof}

Suppose $(T_1,T_2,T_1T_2)$ on $\HS$ be a triangular $\E$-contraction for a pair of commuting contractions $T_1,T_2 \in \mathcal B(\HS)$. Suppose $(V_1,V_2)$ on $\mathcal K \supseteq \HS$ be Ando's isometric dilation of $(T_1,T_2)$. Clearly the commuting triple $(V_1,V_2,V_1V_2)$ dilates $(T_1,T_2,T_1T_2)$ and by Lemma \ref{lem:char-triang1}, $(V_1, V_2,V_1V_2)$ is a triangular $\E$-isometry.

Now suppose $(T_1,T_2,T)$, where $T=T_1T_2$, is a pure triangular $\E$-contraction. Then $T$ is a pure contraction and following the argument as above $(V_1,V_2,V)$ on $\mathcal K$ is a triangular $\E$-isometric dilation of $(T_1,T_2,T)$, where $V=V_1V_2$. Then by Proposition \ref{prop:minimal-dil}, there is a Hilbert space $\widetilde{\mathcal K} \subseteq \mathcal K$ such that $\widetilde{\mathcal K}$ is a joint-invariant subspace of $(V_1,V_2,V)$ and that $(V_1|_{\widetilde{\mathcal K}}, V_2|_{\widetilde{\mathcal K}}, V|_{\widetilde{\mathcal K}})$ is a minimal triangular $\E$-isometric dilation of $(T_1,T_2,T)$. If we denote $(\widetilde{V}_1, \widetilde{V}_2, \widetilde{V})= (V_1|_{\widetilde{\mathcal K}}, V_2|_{\widetilde{\mathcal K}}, V|_{\widetilde{\mathcal K}})$, it follows from Proposition \ref{prop:dilation-extn} that $(\widetilde{V}_1^*, \widetilde{V}_2^*, \widetilde{V}^*)$ is an $\E$-co-isometric extension of $(T_1^*,T_2^*,T^*)$. By Theorem \ref{thm:ti}, $(\widetilde{V}_1, \widetilde{V}_2, \widetilde{V})$ admits a Wold decomposition say $(U_1 \oplus X_1 , U_2 \oplus X_2 , U \oplus X)$ with respect to the orthogonal decomposition $\widetilde{\mathcal K} = \mathcal K_1 \oplus \mathcal K_2$, where $(U_1,U_2,U)$ on $\mathcal K_1$ is a triangular $\E$-unitary and $(X_1,X_2,X)$ on $\mathcal K_2$ is a pure triangular $\E$-isometry. We show that $\HS \subseteq \mathcal K_2$. Then it will follow that $(X_1^*,X_2^*,X^*)$ is an $\E$-co-isometric extension of $(T_1^*,T_2^*,T^*)$ which further implies that $(X_1,X_2,X)$ is a triangular $\E$-isometric dilation of $(T_1,T_2,T)$. If $\HS$ is not a subspace of $\mathcal K_2$, then there is $h\in \HS \cap \mathcal K_1$ such that $T^*h=U^*h$. Now the purity of $T$ implies that ${T^*}^nh \rightarrow 0$ as $n \rightarrow \infty$. Then it follows that ${U^*}^nh \rightarrow 0$ as $n \rightarrow \infty$, which is a contradiction. Hence $\HS \subseteq \mathcal K$ and the proof is complete.

\end{proof}

A pair of commuting contractions $(T_1,T_2)$ is said to be a \textit{pure pair} if their product $T_1T_2$ is a pure contraction. As a straight consequence of Theorem \ref{thm:triang-dilation}, we obtain the following corollary.

\begin{cor}
Every pure pair of contractions $(T_1,T_2)$ dilates to a pure pair of commuting isometries $(V_1,V_2)$.
\end{cor}

Now we present an alternative proof to the well-known result from \cite{D:S:S} which provides factorization of a pure contraction.

\begin{thm} \label{thm:factor}
Let $T$ be a pure contraction on a Hilbert space $\HS$. Then $T=T_1T_2$ for a pair of commuting contractions $T_1,T_2 \in \mathcal B(\HS)$ if and only if there exist a Hilbert space $E$, a projection $P$ and a unitary $U$ in $\mathcal B(E)$ such that $\HS \subseteq H^2(E)$ is a common invariant subspace for $(M_{\Phi}^*,M_{\Psi}^*)$ and the following identities hold:
\begin{itemize}

\item[(i)] $P_{\HS} M_{\Phi \Psi}|_{\HS}= P_{\HS} M_{\Psi \Phi}|_{\HS} = P_{\HS} M_{z}|_{\HS}$

\item[(ii)] $(T_1,T_2) \equiv (P_{\HS} M_{\Phi}|_{\HS}, P_{\HS} M_{\Psi}|_{\HS})$,

\end{itemize}
where $\Phi(z)=(P+zP^{\perp})U^*$ and $\Psi(z)=U(P^{\perp}+zP)$.

\end{thm}

\begin{proof}

Suppose $T=T_1T_2$ for a pair of commuting contractions $T_1,T_2 \in \mathcal B(\HS)$. Then $(T_1, T_2, T)$ is a triangular $\E$-contraction. The desired conclusion follows from Theorem \ref{thm:triang-dilation} and the second part of its proof. The converse part is trivial.

\end{proof}

The set $\{(x_1,x_2,x_3) \in \E : x_1x_2=x_3 \}$ consisting of triangular points of $\E$ is nothing but the intersection of $\E$ with the variety $\{ (x_1,x_,x_3)\in \C^3: x_1x_2=x_3 \}$. Hence we call it the \textit{triangular variety} of the tetrablock. We conclude this Section with the following observation that the triangular variety is not only a spectral set but also a complete spectral set for a triangular $\E$-contraction.

\begin{prop} \label{thm:TR-Variety}
The Triangular variety is a complete spectral set for any triangular $\E$-contraction.
\end{prop}

\begin{proof}

Suppose $(P,Q,PQ)$ is a triangular $\E$-contraction. If $(U_1,U_2)$ is a commuting unitary dilation for $(P,Q)$, then $(U_1,U_2,U_1U_2)$ is a triangular $\E$-unitary dilation of $(P,Q, PQ)$. Needless to mention that $\sigma_T(U_1,U_2,U_1U_2)$ consists of points of the form $(z_1,z_2,z_1z_2)$ with $|z_1|=|z_2|=1$ and such points live on $\partial W = \ov{W} \cap \partial \ov{\E}= \ov{W} \cap \partial b \ov{\E}$, where $W$ is the set of triangular points in $\E$. Thus, $(P,Q,PQ)$ dilates to the boundary of $W$ and consequently $\ov{W}=W \cup \partial W$ is a complete spectral set for $(P,Q,PQ)$ by Arveson's theorem.

\end{proof}


\section{DILATION AND DISTINGUISHED VARIETIES IN THE TETRABLOCK} \label{sec:4}

\vspace{0.4cm}

\noindent In this Section, we are going to show how the existence of $\E$-unitary dilation of an $\E$-contraction $(A,B,P)$ determines a distinguished variety in the tetrablock and vice-versa when $\mathcal D_{p^*}$ is finite dimensional. We first find a necessary and sufficient condition under which an $\E$-contraction $(A,B,P)$ dilates to an $\E$-unitary on the minimal unitary dilation space of $P$. Before that we state a result from \cite{Bh-Sau1} which will be useful.  

\begin{lem}[\cite{Bh-Sau1}, Lemma 11] \label{lem:Sau-Bh}
Let $(A,B,P)$ on $\HS$ be an $\E$-contraction and let $F_1,F_2$ and $G_1,G_2$ be the fundamental operators of $(A,B,P)$ and $(A^*,B^*,P^*)$ respectively. Then the following are equivalent.
\begin{itemize}

\item[(i)] $[F_1,F_2]=0$ and $[F_1^*,F_1]=[F_2^*,F_2]$ ;

\item[(ii)] $[G_1,G_2]=0$ and $[G_1^*,G_1]=[G_2^*,G_2]$.

\end{itemize}

\end{lem}

\begin{thm} \label{thm:E-unitary dilation}
Let $(A,B,P)$ be an $\E$-contraction on a Hilbert space $\HS$ and let $F_1,F_2$ and $G_1,G_2$ be the fundamental operators of $(A,B,P)$ and $(A^*,B^*,P^*)$. Let $\mathcal K$ be the minimal unitary dilation space for the contraction $P$. Then $(A,B,P)$ dilates to an $\E$-unitary $T_1, T_2, U$ on $\mathcal K$ with $U$ being the minimal unitary dilation of $P$ if and only if $[F_1,F_2]=0$ and $[F_1^*,F_1]=[F_2^*,F_2]$.
\end{thm}

\begin{proof}

As varient of this theorem was proved in \cite{Bh-Sau1} (see Theorems 4 \& 15 in \cite{Bh-Sau1}). We provide an alternative proof to the forward direction here. Suppose $(A,B,P)$ is an $\E$-contraction and its fundamental operators $F_1,F_2$ satisfy $[F_1,F_2]=0$ and $[F_1^*,F_1]=[F_2^*,F_2]$. Let us define $(T_1,T_2,U)$ on $\mathcal K$ in the following way:
\begin{eqnarray}\label{2.3}
&T_1 =\footnotesize \left[
\begin{array}{ c c c c|c|c c c c}
\bm{\ddots}&\vdots &\vdots&\vdots   &\vdots  &\vdots& \vdots&\vdots&\vdots\\
\cdots&0&F_1&F_{2}^*  &0&  0&0&0&\cdots\\
\cdots&0&0&F_1  &F_{2}^*D_{P}&  -F_{2}^*P^*&0&0&\cdots\\
\hline

\cdots&0&0&0   &A&   D_{P^*}{G_{2}}&0&0&\cdots\\ \hline

\cdots&0&0&0   &0&  {G_1}^*& {G_{2}}&0&\cdots\\
\cdots&0&0&0   &0&  0&{G_1}^*&{G_{2}}&\cdots\\
\vdots&\vdots&\vdots&\vdots&\vdots&\vdots&\vdots&\vdots&\bm{\ddots}\\
\end{array} \right]\,, \\&
T_2 =\footnotesize \left[
\begin{array}{ c c c c|c|c c c c}
\bm{\ddots}&\vdots &\vdots&\vdots   &\vdots  &\vdots& \vdots&\vdots&\vdots\\
\cdots&0& F_2 & F_{1}^*  &0&  0&0&0&\cdots\\
\cdots&0&0& F_2  & F_{1}^*D_{P}&  -F_{1}^*P^*&0&0&\cdots\\
\hline

\cdots&0&0&0   &B&   D_{P^*}{G_{1}}&0&0&\cdots\\ \hline

\cdots&0&0&0   &0&  {G_2}^*& {G_{1}}&0&\cdots\\
\cdots&0&0&0   &0&  0&{G_2}^*&{G_{1}}&\cdots\\
\vdots&\vdots&\vdots&\vdots&\vdots&\vdots&\vdots&\vdots&\bm{\ddots}\\
\end{array} \right] \,,
\\&
\text{ \large and }\quad U = \left[
\begin{array}{ c c c c|c|c c c c}
\bm{\ddots}&\vdots &\vdots&\vdots   &\vdots  &\vdots& \vdots&\vdots&\vdots\\
\cdots&0&0&I  &0&  0&0&0&\cdots\\
\cdots&0&0&0  &D_{P}&  -{P}^*&0&0&\cdots\\ \hline

\cdots&0&0&0   &P&   D_{P^*}&0&0&\cdots\\ \hline

\cdots&0&0&0   &0&  0& I&0&\cdots\\
\cdots&0&0&0   &0&  0&0&I&\cdots\\
\vdots&\vdots&\vdots&\vdots&\vdots&\vdots&\vdots&\vdots&\bm{\ddots} \label{2.33}\\
\end{array} \right].
\end{eqnarray}
It is evident from the block matrix form that $U$ is the minimal Schaeffer unitary dilation of $P$. We show that $(T_1,T_2,U)$ is an $\E$-unitary dilation of $(A,B,P)$. Suppose the block matrix of $T_1, T_2$ and $U$ with respect to the decomposition $l^2(\mathcal
D_{P})\oplus \mathcal H \oplus l^2(\mathcal D_{P^*})$ of $\mathcal
K$ are
\begin{equation}\label{sec:eqn-011}
\left[
\begin{array}{ccc}
R_{1} & R_{2} & R_{3}\\
0 & A & R_{4}\\
0& 0& R_{5}  \end{array} \right],
\left[
\begin{array}{ccc}
S_{1} & S_{2} & S_{3}\\
0 & B & S_{4}\\
0& 0& S_{5}  \end{array} \right]
\text{ and } \left[
\begin{array}{ccc}
Q_{1} & Q_{2} & Q_{3}\\
0 & P & Q_{4}\\
0& 0& Q_{5}  \end{array} \right]
\end{equation}
respectively. The commutativity of the triple $(T_1,T_2,U)$ was proved in Theorem 4 in \cite{Bh-Sau1}. It is obvious from the upper triangular forms of the block matrices of $T_1,T_2,U$ that
\[
 P_{\mathcal H}(T_1^{n_1},T_{2}^{n_2}U^n)|_ {\mathcal H}=A^{n_1}B^{n_2}P^n\,,
\]
for all integers $n_1,n_2,n$. This proves that
$(T_1,T_2,U)$ dilates $(A,B,P)$. The
minimality of the dilation follows from the fact that $U$ on $\mathcal K$
is the minimal unitary dilation of $P$. Thus, it is enough to show that $(T_1,T_2,U)$ is an $\E$-unitary which by Theorem \ref{thm:tu} is same as proving that
$(T_1,T_{2},U)$ is a $\Gamma_n$-contraction because $U$ is already a unitary. Since $\ov{\E}$ is polynomially convex, it suffices to verify von Neumann's inequality for any $f \in \C[z_1,z_2,z_3]$, that is,
\[
\| f(T_1,T_{2},U) \| \leq \| f \|_{\infty , \E}.
\]
The fact that $\sigma_T(T_1,T_2,U) \subseteq \ov{\E}$ then follows as a consequence. Note that we have from the proof of Theorem 4 in \cite{Bh-Sau1} that $T_1=T_{2}^*U$ and $T_2=T_1^*U$. Now
\[
T_1^*T_1=U^*T_2T_1=U^*T_1T_2=T_1U^*T_2=T_1T_1^*.
\]
similarly we have $T_2^*T_2=T_2T_2^*$ and so $T_1,T_2, U$ are commuting normal operators.\\

Let $f$ be any polynomial in $\C[z_1,z_2,z_3]$. It is evident from the block matrices of $T_1,T_{2},U$
that
\[
 f(T_1,T_{2},U)  =\left[
\begin{array}{ccc}
f(R_{1}, S_{1}, Q_1) & * & *\\
0 & f(A,B,P) & *\\
0& 0& f(R_{5},S_{5},Q_5)  \end{array} \right].
\]
Note that $f(T_1,T_{2},U)$ is a normal operator as $T_1,T_2,U$ are commuting normal operators. Therefore,
\[
\|f(T_1,T_2,U)\|=r(f(T_1,T_{2},U)).
\]
Let us consider the following diagonal block matrices of operators defined on $\mathcal K$:
\[
\widehat{T_1}=\left[
\begin{array}{ccc}
R_{1} & 0 & 0\\
0 & A & 0\\
0& 0& R_{5}  \end{array} \right], \;
\widehat{T_2}=\left[
\begin{array}{ccc}
S_{1} & 0 & 0\\
0 & B & 0\\
0& 0& S_{5}  \end{array} \right], \;\widehat{U}=\left[
\begin{array}{ccc}
Q_1 & 0 & 0\\
0 & P & 0\\
0& 0& Q_5  \end{array} \right].
\]
Since $F_1,F_2$ and $G_1,G_2$ are fundamental operators of $(A,B,P)$ and $(A^*,B^*,P^*)$ respectively, we have that the numerical radii $\omega(F_1^*+zF_2), \omega(G_1^*+zG_2)$ are not greater than $1$ for any $z\in \T$. Again $F_1^*+zF_2$ is a normal operator for any $z\in \T$ as $F_1,F_2$ commute and $[F_1^*,F_1]=[F_2^*,F_2]$. Invoking Lemma \ref{lem:Sau-Bh} we have that $G_1^*+zG_2$ is normal for any $z$ of unit modulus. Therefore,
\[
\|F_1^*+zF_2\|_{\infty, \T} \leq 1 ,\; \|G_1^*+zG_2\|_{\infty, \T} \leq 1.
\] 
Thus, it follows from Theorem \ref{modelthm1} that $(R_{1},S_{1},Q_1)$ and $(R_{5},S_{5},Q_5)$ are $\E$-isometry and $\E$-co-isometry respectively. Therefore,
$(\widehat{T_1},\widehat{T_{2}},\widehat{U})$ is an
$\E$-contraction by being direct sum of three
$\E$-contractions. Note that
\[
 f(\widehat{T_1},\widehat{T_{2}},\widehat{U})  =\left[
\begin{array}{ccc}
f(R_{1}, S_{1}, Q_1) & 0 & 0\\
0 & f(A,B,P) & 0\\
0& 0& f(R_{5},S_{5},Q_5)  \end{array} \right].
\]
We now apply Lemma 1 of \cite{hong} to $f(T_1,T_{2},U)$, which states that the spectrum of an operator of the
form $\begin{bmatrix} X&Y\\0&Z
\end{bmatrix}$ is a subset of $\sigma(X)\cup \sigma (Z)$.
So, we arrive at
\begin{align*}
\sigma(f(T_1,T_{2},U)) & \subseteq
\sigma(f(R_{1},S_{1}, Q_1))\cup
\sigma(f(A,B,P)) \\& \quad \cup
\sigma(f(R_{5},S_{5}, Q_5)) \\
& =
\sigma(f(\widehat{T_1},\widehat{T_{2}},\widehat{U})).
\end{align*}
By an application of the spectral mapping theorem we have
\[
\sigma(f(\widehat{T_1},\widehat{T_{2}},\widehat{U}))
=\{ f(\lambda_1,\lambda_2,\lambda_3)\,:\,
(\lambda_1,\lambda_2,\lambda_3)\in \sigma_T(\widehat{T_1},\widehat{T_{2}}, \widehat{U}) \}.
\]
Since $(\widehat{T_1},\widehat{T_{2}},\widehat{U})$
is an $\E$-contraction,
$\sigma_T(\widehat{T_1},\widehat{T_{2}}, \widehat{U})
\subseteq \ov{\E}$. So, we have
\[
r(f(\widehat{T_1},\widehat{T_{2}}, \widehat{U}))\leq
\sup_{(z_1,z_2,z_{3})\in \ov{\E}} \,|f(z_1,z_2,z_{3})|=\|f\|_{\infty,\ov{\E}}.
\]
Since
$
\sigma(f(T_1,T_{2},U)) \subseteq
\sigma(f(\widehat{T_1},\widehat{T_{2}}, \widehat{U}))\,,
$
we have that
$r(f(T_1,T_2,U)) \leq \|f\|_{\infty, \ov{\E}}$. Therefore, we have
$
\|f(T_1,T_2,U)
\|=r(f(T_1,T_2,U))\leq
\|f\|_{\infty, \ov{\E}},
$
and the proof is complete.

\end{proof}

Recall that a distinguished variety in $\E$ is the intersection of an algebraic variety $V$ with $\E$ such that $V$ exits the tetrablock through the distinguished boundary $b\E$ without intersection any other part of the topological boundary of $\E$.
\begin{defn}
A pair of commuting matrices $F,G$ is said to define a distinguished variety in $\E$ if $\mathcal Z(f_1,f_2)\cap \E$ is a distinguished variety in $\E$, where $\mathcal Z(f_1,f_2)$ is the algebraic variety generated by the polynomials $\{f_1,f_2 \}= \{\det\,(F^*+z_3G -z_1I), \, \det \, (G^*+z_3F - z_2I) \}$.
\end{defn}
The following theorem provides a description of all distinguished varieties in the tetrablock.
\begin{thm}[\cite{S:P-tetra2}, Theorem 4.5] \label{thm:DVchar-T}
Let
\begin{equation}\label{eq:W} \Omega = \{(x_1,x_2,x_3) \in
\mathbb E \,:\, (x_1,x_2)\in \sigma_T(A_1^*+x_3A_2,
A_2^*+x_3A_1)\},
\end{equation}
where $A_1,A_2$ are commuting square matrices of same order such
that
\begin{enumerate}
\item $[A_1^*,A_1]=[A_2^*,A_2]$ \item $\|A_1^*+ zA_2\|_{\infty,
\mathbb T}<1$.
\end{enumerate}
Then $\Omega$ is a one-dimensional distinguished variety in
$\mathbb E$. Conversely, every distinguished variety in $\mathbb
E$ is one-dimensional and can be represented as (\ref{eq:W}) for
two commuting square matrices $A_1,A_2$ of same order, such that
\begin{enumerate}
\item $[A_1^*,A_1]=[A_2^*,A_2]$ \item $\|A_1^*+ zA_2 \|_{\infty,
\mathbb T}\leq 1$.
\end{enumerate}
\end{thm}

\begin{prop} \label{thm:pure-dilation}
Let $(A,B,P)$ on $\HS$ be a pure $\E$-contraction and let $(G_1,G_2)$ be the $\ft$-tuple of $(A^*,B^*,P^*)$. Then the following are equivalent.
\begin{enumerate}

\item $(A,B,P)$ dilates to an $\E$-unitary $(T_1,T_2,U)$ on $L^2(\mathcal D_{P^*})$, where $U$ is the minimal unitary dilation of $P$.

\item $[G_1,G_2]=0$, $[G_1^*,G_1]=[G_2^*,G_2]$.

\end{enumerate}

\end{prop}

\begin{proof}

\textbf{(1) $\Rightarrow$ (2).} Since $P$ is a pure contraction, it follows from Nagy-Foias theory (see \cite{nagy}) that upto a unitary $M_z$ on $L^2(\mathcal D_{P^*})$ is the minimal unitary dilation of $P$. Thus, it follows from Theorem \ref{thm:E-unitary dilation} that
$[G_1,G_2]=0$, $[G_1^*,G_1]=[G_2^*,G_2]$.\\

\noindent \textbf{(2) $\Rightarrow$ (1).} Let us assume (2). Then by Theorem \ref{thm:pure-E-dilation}, $(T_{G_1^*+ zG_{2}},T_{G_{2}^*+zG_{1}},T_z)$ on $H^2(\mathcal D_{P^*})$ is a minimal pure $\E$-isometric dilation of $(A,B,P)$. The $\E$-isometry $(T_{G_1^*+zG_{2}},T_{G_{2}^*+zG_{1}},T_z)$ naturally extends to the $\E$-unitary $(M_{G_1^*+zG_{2}},M_{G_{2}^*+zG_{1}},M_z)$ on $L^2(\mathcal D_{P^*})$. Therefore, $(M_{G_1^*+zG_{2}},M_{G_{2}^*+zG_{1}},M_z)$ on $L^2(\mathcal D_{P^*})$ is an $\E$-unitary dilation of $(A,B,P)$.

\end{proof}

We now present the first interplay between dilation of a pure $\E$-contraction and distinguished varieties in the tetrablock.

\begin{thm} \label{thm:dilation-variety1}

Let $\Upsilon=(A,B,P)$ be a pure $\E$-contraction on $\HS$ with $(G_1,G_2)$ being the $\ft$-tuple of $(A^*,B^*,P^*)$. If $\mathcal D_{P^*}$ is finite dimensional and $\|G_1^*+zG_2 \|_{\infty, \, \T}< 1$, then the following are equivalent.

\begin{enumerate}

\item $(A,B,P)$ possesses an $\E$-unitary dilation $(T_1, T_2,U)$ on $L^2(\mathcal D_{P^*})$, where $U$ is the minimal unitary dilation of $P$.\\

\item $[G_1,G_2]=0$, $[G_1^*,G_1]=[G_2^*,G_2]$. \\

\item The set of polynomials $\{f_1, f_2 \}= \{\det\, (G_1^*+z_3 G_{2}-z_1I),\, \det\, (G_2^*+z_3 G_{1}-z_2I) \}$ defines a distinguished variety $\Omega_\Upsilon$ in $\E$.\\

\item $(A,B,P)$ has a normal $\partial \ov{\Omega}_\Upsilon -$dilation on $L^2(\mathcal D_{P^*})$, where $\Omega_\Upsilon = \mathcal Z(f_1,f_2)\cap \E$ and $\partial \ov{\Omega}_{\Upsilon}=b\ov{\E} \cap \ov{\Omega}_{\Upsilon} $.

\end{enumerate}

\end{thm}

\begin{proof}

\textbf{(1) $\Rightarrow$ (2)} Since $P$ is a pure contraction, we have from Nagy-Foias theory (see \cite{nagy}) that upto a unitary $M_z$ on $L^2(\mathcal D_{P^*})$ is the minimal unitary dilation of $P$. Thus, it follows from Lemma \ref{lem:Sau-Bh} and Theorem \ref{thm:E-unitary dilation} that $[G_1,G_2]=0$ and $[G_1^*,G_1]=[G_2^*,G_2]$.\\

\noindent \textbf{(2) $\Rightarrow$ (3)}  Since $\mathcal D_{P^*}$ is finite dimensional, $G_1,G_2$ are commuting matrices. It follows from Theorem \ref{thm:DVchar-T} that the set
\[
\Omega = \{ (x_1,x_2,x_3)\in \E \,:
\; (x_1,x_2) \in \sigma_T(G_1^*+x_3 G_{2}\,,\,
G_2^*+x_3G_1) \}
\]
is a distinguished variety in $\E$ and following the proof of Theorem \ref{thm:DVchar-T} from \cite{S:P-tetra2} we see that $\Omega$ is determined by the set of polynomials
\[
\{f_1, f_2 \}= \{\det\, (G_1^*+z_3 G_{2}-z_1I),\, \det\, (G_2^*+z_3 G_{1}-z_2I) \}.
\]

\noindent \textbf{(3) $\Rightarrow$ (4)} It is obvious that $(x_1,x_2) \in \sigma_T(G_1^*+x_3G_2,G_2^*+x_3G_1)$ if and only if $(x_1,x_2,x_3) \in \sigma_T(G_1^*+x_3G_2,G_2^*+x_3G_1, x_3I)$. If the set of polynomials $\mathcal F=\{f_1, f_2\}$ defines a distinguished variety $\Omega_\Upsilon$ in $\E$, then
\[
\Omega_\Upsilon = V_{\mathcal F} \cap \E = \{ (x_1,x_2,x_3)\in \E \,:\, (x_1,x_2) \in \sigma_T(G_1^*+x_3G_2, \, G_2^*+x_3G_1)\}
\]
and by Theorem \ref{thm:DVchar-T}, $G_1,G_2$ satisfy
$[G_1,G_2]=0$, $[G_1^*,G_1]=[G_2^*,G_2]$ and $\|G_1^*+G_2z \|_{\infty , \, \T} \leq 1$. It follows from Proposition \ref{thm:pure-dilation} that $(A,B,P)$ dilates to the $\E$-unitary $\left( M_{G_1^*+zG_2}, M_{G_2^*+zG_1}, M_z \right)$ on $L^2(\mathcal D_{P^*})$. In order to establish that $(A,B,P)$ possesses a normal $\partial \ov{\Omega}_\Upsilon -$dilation, it suffices if we show that $\sigma_T\left( M_{G_1^*+zG_2}, M_{G_2^*+zG_1}, M_z \right) \subseteq b\ov{\Omega}_\Upsilon$. Since $\mathcal F = \{ f_1,f_2 \}$ defines the distinguished variety $\Omega_{\Upsilon}$ in $\E$, we have that $\partial\ov{\Omega}_\Upsilon = V_{\mathcal F} \cap \partial \ov{\E} \subseteq b \ov{\E}$. Suppose $|y_3|=1$ and $(y_1,y_2,y_3) \in \sigma_T(G_1^*+y_3G_2,G_2^*+y_3G_1, y_3I)$. Suppose $\xi$ is a unit joint eigenvector of $(G_1^*+y_3G_2,G_2^*+y_3G_1, y_3I)$ with respect to the joint eigenvalue $(y_1,y_2,y_3)$. Then
\[
(G_1^*+y_3G_2) \xi = y_1 \xi \quad \& \quad (G_2^*+y_3G_1) \xi = y_2 \xi
\]
and taking inner product with $\xi$ we have that
$y_1=\alpha_1 +\ov{\alpha}_2 y_3$, $y_2= \alpha_2 + \ov{\alpha}_1y_3$, where $\alpha _i = \la G_i^* \xi, \xi \ra$. Since $\left( M_{G_1^*+zG_2}, M_{G_2^*+zG_1}, M_z \right)$ is an $\E$-unitary, both $\| G_1^*+zG_2 \|_{\infty, \T}, \|G_2^*+G_1z \|_{\infty, \T}$ are not greater than $1$. Here $\alpha_1$ and
$\alpha_2$ are unique because we have that $
y_1-\bar{y_2}y_3=\alpha_1(1-|y_3|^2)$ and $
y_2-\bar{y_1}y_3=\alpha_2(1-|y_3|^2) $ which lead to
\[
\alpha_1=\frac{y_1-\bar{y_2}y_3}{1-|x_3y|^2} \text{ and }
\alpha_2=\frac{y_2-\bar{y_1}y_3}{1-|y_3|^2}.
\]
Since $G_1^*+zG_2$ is a normal matrix for every $z \in \T$, we have that $\|G_1^*+zG_2\|=\omega(G_1^*+zG_2)\leq 1$ and thus $\omega(z_1G_1^*+z_2G_2)\leq 1$, for
every $z_1,z_2$ in $\mathbb T$. Therefore,
\[
|z_1\langle G_1^*\xi,\xi \rangle +z_2 \langle G_2\xi,\xi
\rangle| \leq 1, \text{ for every } z_1,z_2 \in \mathbb T.
\]
If both $ \langle G_1^*\xi,\xi \rangle $ and $\langle G_2 \xi,\xi \rangle$ are non-zero, we can choose $z_1=\frac{|\langle
G_1^*\xi,\xi \rangle|}{\langle G_1^*\xi,\xi \rangle}$ and
$z_2=\frac{|\langle G_2 \xi,\xi \rangle|}{\langle G_2 \xi,\xi \rangle}$ to get $|\alpha_1|+|\alpha_2| \leq 1$. If any of them or both
$ \langle G_1^*\xi,\xi \rangle $ and $\langle G_2 \xi,\xi
\rangle$ are zero, then also $|\alpha_1|+|\alpha_2|\leq 1$. Therefore, by Theorem \ref{thm:21}, $(y_1,y_2,y_3)$ is in $\ov{\E}$. Since $|y_3|=1$, we have that $(y_1,y_2,y_3)\in b \ov{\E}$. Also, $(y_1,y_2,y_3)\in V_{\mathcal F}$.  Thus, $(y_1,y_2,y_3)\in b \ov{\E} \cap \ov{\Omega}_\Upsilon = \partial \ov{\Omega}_\Upsilon$. Therefore,
\[
\{ (y_1,y_2,y_3)\in \sigma_T(G_1^*+y_3G_2, G_2^*+y_3G_1, y_3I)\,:\, |y_3|=1  \} \subseteq \partial \ov{\Omega}_\Upsilon.
\]
From here we have that
\[
\sigma_T(G_1^*+zG_2,G_2^*+zG_1,zI)=\sigma_T \left( M_{G_1^*+zG_2}, M_{G_2^*+zG_1}, M_z \right) \subseteq \partial \ov{\Omega}_\Upsilon,
\]
and we are done.\\

\noindent \textbf{(4) $\Rightarrow$ (1)} Note that a normal $\partial \ov{\Omega}_\Upsilon-$dilation is an $\E$-unitary dilation, because, $\partial \ov{\Omega}_\Upsilon \subseteq b\ov{\E}$. Thus, (1) follows trivially from (4).

\end{proof}

The following result is a next step to Theorem \ref{thm:dilation-variety1} in the sense that here we remove the purity condition on $P$ and present a more generalized interplay between dilation and distinguished varieties.

\begin{thm} \label{thm:dilation-variety2}

Let $(A,B,P)$ be an $\E$-contraction on $\HS$ with $\mathcal D_P$ being finite dimensional. Suppose $F_1,F_2$ are the fundamental operators of $(A,B,P)$ satisfying $\| F_1^*+zF_2 \|_{\infty, \, \T}<1$. Then the following are equivalent.

\begin{enumerate}

\item $(A,B,P)$ dilates to an $\E$-unitary $(T_1,T_2,U)$ on $l^2(\mathcal D_{P^*})\oplus \HS \oplus l^2(\mathcal D_P)$, where $U$ is the minimal unitary dilation of $P$.\\

\item $F_1,F_2$ commute, $[F_1^*,F_1]=[F_2^*,F_2]$.\\

\item The set of polynomials $\{f_i=\det\, (F_1+z_3F_2^*-z_iI):\, 1 \leq i \leq n-1 \}$ defines a distinguished variety in $\E$.

\end{enumerate}

\end{thm}

\begin{proof}

Evidently \textbf{(1) $\Leftrightarrow$ (2)} follows from Theorem \ref{thm:E-unitary dilation} and \textbf{(2) $\Leftrightarrow$ (3)} follows from Theorem \ref{thm:dilation-variety1}. Note that when we apply Theorem \ref{thm:dilation-variety1} to establish $(2) \Leftrightarrow (3)$, we do not need the purity assumption on $P$ and it is evident from the proof of Theorem \ref{thm:dilation-variety1}.

\end{proof}



\end{document}